[1994/12/01]
\documentclass[draft]{article}
\title{\textbf{Some rigidity results related to the Obata type equation}}
\author{Yiwei Liu and Yi-Hu Yang\footnote{Partially supported by NSF of China (No. 12071283)}}

\date{\today}

\usepackage{amsmath,amsthm,amssymb}
\usepackage[all,pdf]{xy}
\chardef\bslash=`\\ 

\newtheorem{theorem}{\bf Theorem}[section]
\newtheorem{lemma}[theorem]{\bf Lemma}
\newtheorem{proposition}[theorem]{\bf Proposition}

\newtheorem{corollary}[theorem]{\bf Corollary}

\newenvironment{proof1}{\noindent{\em Proof of Proposition  \ref{theorem: 1.1} (1):}}{\quad \hfill$\Box$\vspace{2ex}}
\newenvironment{proof2}{\noindent{\em Proof of Theorem  \ref{theorem: 1.2}:}}{\quad \hfill$\Box$\vspace{2ex}}
\newenvironment{proof3}{\noindent{\em Proof of Theorem  \ref{theorem: 1.3}:}}{\quad \hfill$\Box$\vspace{2ex}}
\newenvironment{proof4}{\noindent{\em Proof of Theorem  \ref{theorem: 1.4}:}}{\quad \hfill$\Box$\vspace{2ex}}
\newenvironment{proof5}{\noindent{\em Proof of Proposition  \ref{theorem: 1.1} (2):}}{\quad \hfill$\Box$\vspace{2ex}}

\newenvironment{remark}{\noindent{\bf Remark.}}{\vspace{2ex}}

\newcommand{\eval}[2][\right]{\relax
  \ifx#1\right\relax \left.\fi#2#1\rvert}

\begin{document}
\maketitle

\renewcommand{\sectionmark}[1]{}

\begin{abstract}
    Let $(\Omega^{n+1},g)$ be an $(n + 1)$-dimensional smooth complete connected Riemannian manifold 
    with compact boundary $\partial\Omega=\Sigma$ and $f$ a smooth function on $\Omega$ which satisfies 
    the Obata type equation $\nabla^2 f -fg =0$ with Robin boundary condition $f_{\nu} = cf$, 
    where $c=\coth{\theta}>1$. In this paper, we provide some rigidity results based on the warped product structure of $\Omega$ determined by the equation $\nabla^2 f -fg =0$ and appropriate curvature assumptions. We also apply a similar method to the Obata type equation $\nabla^2 f +fg =0$ 
    and get a rigidity result on the standard sphere $\mathbb{S}^{n+1}$.
\end{abstract}

\section{Introduction}

Let $(\Omega^{n+1},g)$ be an $(n + 1)$-dimensional ($n\geq 2$) smooth complete connected Riemannian manifold with compact boundary $\partial\Omega =\Sigma$. In this paper, we concern the Obata type equation $\nabla^2 f -fg =0$ with Robin boundary condition $f_{\nu}=cf$ and related rigidity of $\Omega$. 

For a complete (compact) Riemannian manifold $(M,g)$ with boundary or without boundary, one can consider 
the following Obata type equation
$$
\nabla^2 f + kfg = 0,
$$
where $k\in \{1,0,-1\}$. 

For the case $k=1$ and manifolds without boundary, the above equation 
is related to many rigidity results. In \cite{lichnerowicz1958geometrie}, Lichnerowicz proved 
that the first eigenvalue $\lambda_1$ of a closed manifold of dimension $n$ whose Ricci curvature has
lower bound $n-1$ is at least $n$, that is, $\lambda_1 \geq n$. Accordingly, Obata \cite{MR0142086} 
proved that if the equality holds, then the manifold must be isometric to a standard sphere 
by using the following equation
$$
\nabla^2 f + fg = 0.
$$

For a compact manifold $(\Omega ^{n+1} , g)$ with boundary $\partial\Omega=\Sigma$, 
one also can consider the Obata type equation with various boundary conditions. 
In \cite{reilly1977applications}, Reilly obtained the Lichnerowicz-type lower bound 
for the first Dirichlet eigenvalue $\mu_1$ on $\Omega $ with $\mbox{Ric}^{\Omega} \geq n $ 
and the boundary $\Sigma$ being of non-negative mean curvature, that is, $\mu_1 \geq n+1$. 
In particular, he also proved that $\mu_1 = n+1$ if and only if $\Omega $ is isometric to 
the standard hemisphere. Here, the proof of the rigidity result depends again 
on the following Obata type equation with Dirichlet boundary condition
$$
\left\{
\begin{aligned}
&\nabla^2f +fg =0,\,\,{\text{in}}\,\, \Omega,\\
&f=0,\,\,\,\,\,\,\,\,\,\,\,\,\,\,\,\,\,\,\,\,\,\,\,{\text{on}} \,\, \Sigma.
\end{aligned}
\right.
$$
Later, Escobar \cite{MR1072395} and Xia \cite{MR1226678} independently proved that the first Neumann eigenvalue $\eta_1$ on $\Omega $ satisfies $\eta_1 \geq n+1 $ if $\mbox{Ric}^{\Omega} \geq n$ and the boundary is convex. Moreover, $\eta_1 = n+1$ if and only if $\Omega $ is isometric to the standard hemisphere and the proof of this rigidity result also depends on the Obata type equation with Neumann boundary condition
$$
\left\{
\begin{aligned}
&\nabla^2f +fg =0,\,\,{\text{in}}\,\, \Omega,\\
&f_{\nu}=0,\,\,\,\,\,\,\,\,\,\,\,\,\,\,\,\,\,\,\,\,\,{\text{on}} \,\, \Sigma,
\end{aligned}
\right.
$$
where $\nu$ is the outward unit normal. Recently, Chen, Lai and Wang \cite{chen2020obata} studied the Obata type equation with Robin boundary condition
$$
\left\{
\begin{aligned}
&\nabla^2f +fg =0,\,\,{\text{in}}\,\, \Omega,\\
&f_{\nu}-cf=0,\,\,\,\,\,\,\,{\text{on}} \,\, \Sigma,
\end{aligned}
\right.
$$
and also obtained many rigidity results.

For the case $k=0$, the following boundary value problem
$$
\left\{
\begin{aligned}
&\nabla^2f  =0,\,\,\,\,\,\,\,\,\,\,\,\,\,\,\,{\text{in}}\,\,\Omega,\\
&f_{\nu}-cf=0,\,\,\,\,\,\,\,{\text{on}}\,\, \Sigma
\end{aligned}
\right.
$$
has been further studied, where $c$ is a positive constant. In \cite{raulot2012first}, Raulot and Savo proved that $\Omega$ is isometric to a Euclidean ball with radius $\frac{1}{c}$ if $\Omega$ has 
non-negative sectional curvature, the principal curvatures of the boundary $\Sigma$ are 
bounded from below by $c$, and there exists a non-constant smooth function $f$ satisfying the above boundary value problem. Later, Xia and Xiong \cite{xia2023escobar} extend this result to a compact 
manifold $\Omega$ whose Ricci curvature satisfies $\mbox{Ric}^{\Omega} \geq 0$ and the mean curvature 
$H$ of $\Sigma$ $\geq c$. It is worth noting that their result is used to prove the rigidity part in Escobar’s conjecture for manifolds with non-negative
sectional curvature.

As for the case $k=-1$, Kanai \cite{MR0707845} proved that a complete manifold $(M^n,g)$ is isometric 
to the standard hyperbolic space $\mathbb{H}^n$ if and only if there exists a non-constant function $f$ 
on $M$ with a critical point and satisfying $\nabla^2 f-fg=0$. At the same time, he also discussed 
the case that $f$ does not have critical points. In \cite{MR4080902}, Galloway and Jang focused on the following Dirichlet boundary problem
$$
\left\{
\begin{aligned}
&\nabla^2f -fg =0,\,\,{\text{in}}\,\, \Omega,\\
&f=a,\,\,\,\,\,\,\,\,\,\,\,\,\,\,\,\,\,\,\,\,\,\,\,{\text{on}}\,\, \Sigma,
\end{aligned}
\right.
$$
and proved some rigidity results. For more information about the Obata type equations, interested readers also can refer to \cite{almaraz2017rigidity} and \cite{MR3326231} . 

Very recently, Lai and Zhou \cite{MR4560221} studied the obata type equation $\nabla^2 f -fg=0$ with various boundary conditions and $f$ being of interior critical points, and gave some rigidity results in the standard hyperbolic space $\mathbb{H}^{n+1}$; in particular, they considered the boundary condition
$f_{\nu}=cf$ on $\Sigma$ with $0< c< 1$, and showed that 


{\it If $f$ is of an interior critical point, then $\Omega$ is isometirc to a domain in $\mathbb{H}^{n+1}$ 
bounded by $\mathbb{S}^k(\sqrt{\frac{c^2}{1-c^2}})\times \mathbb{H}^{n-k}(c^2-1)$, for some $k=0,1,\cdots,n$,
where $\mathbb{S}^k(\sqrt{\frac{c^2}{1-c^2}})$ is the sphere of radius $\sqrt{\frac{c^2}{1-c^2}}$ 
and $\mathbb{H}^{n-k}(c^2-1)$ is the simply connected space with constant sectional curvature $c^2-1$.}

In this paper, we concentrate on the equation $\nabla^2f-fg =0$ with the same boundary condition $f_{\nu}=cf$
but $c=\coth{\theta}>1$ with some $\theta >0$, i.e. the following
\begin{equation}\label{eqn-1}
\begin{cases}
\begin{aligned}
&\nabla^2 f- fg=0,& {\text{in}} \,\Omega,\\
&f_{\nu}-cf=0,&{\text{on}} \, \Sigma.
\end{aligned}
\end{cases}
\end{equation}
It should be pointed out that in this case $f$ has 
no critical points (cf. Proposition \ref{prop: 2.1}); due to this reason,  
in order to get some finer rigidity results, we need some additional curvature restrictions. 
To this end, we first introduce the following two curvature assumptions
\\
\\
($K_1$) Let Ric$^{\Omega}$ be the Ricci curvature tensor of $\Omega$ and $H$ be the mean
curvature of $\Sigma$. We assume that
$$
\mbox{Ric}^{\Omega}\geq -n
$$
and
$$H\geq c.$$
\\
($K_2$) Let $S$ be some boundary component, $R^{\Omega}$ be the curvature tensor of $\Omega$, $h$ be the second fundamental form of $S$ and $k\in  \{2,\cdots,n\}$. We assume that
$$
-\sum_{j=2}^{k} R^{\Omega}(e_1,e_j,e_1,e_j) \geq (k-1)(1-\frac{2}{c^2})
$$
for any orthonormal vectors $\{e_1,e_2,\cdots,e_{k}\}$ in $TS$ and
$$h<cg|_{S}.$$

Now, we can state our rigidity results.
Our first result is the following
\begin{theorem}
Let $(\Omega^{n+1}, g)$ be an $(n+1)$-dimensional ($n\geq 2$) smooth complete connected Riemannian manifold with compact boundary $\Sigma$ which satisfies the assumption ($K_1$).  Assume that there exists a non-constant function $f\in \, C^{\infty}(\Omega)$ which is the solution of equation (\ref{eqn-1}), then $\Omega$ is isometric to a geodesic ball
of radius $\tanh^{-1} {(\frac{1}{c})}$ in the hyperbolic space $\mathbb{H}^{n+1}$.
\label{theorem: 1.3}
\end{theorem}

It is worth noting that Theorem \ref{theorem: 1.3} can be seen as an extension of Corollary 4.5 in \cite{Liu} and the idea of the proof comes from Proposition 4.3 in \cite{xia2023escobar}. We also point out that the curvature assumption ($K_1$) in Theorem \ref{theorem: 1.3} implies that $\Omega$ is compact (see Theorem 0.1 in \cite{MR3404748}) and $\Sigma$ is connected (see Lemma 2.1 in \cite{MR3373063}).

Based on the curvature assumption ($K_2$) and a lower bound condition for the diameter of boundary components, we can prove the following
\begin{theorem}
Let $(\Omega^{n+1}, g)$ be an $(n+1)$-dimensional ($n\geq 2$) smooth complete connected Riemannian manifold with compact boundary $\Sigma$ which satisfies the assumption ($K_2$).
Assume that there exists a non-constant function $f\in \, C^{\infty}(\Omega)$ which is the solution of equation (\ref{eqn-1}) and the diameter of any boundary component has a lower bound $\sqrt{\frac{c^2}{c^2-1}}\pi$. Then $\Sigma$ has at most two components and any boundary component is isometric to the sphere $\mathbb{S}^{n}(\sqrt{\frac{c^2}{c^2-1}})$ of radius $\sqrt{\frac{c^2}{c^2-1}}$. In particular, we also have
\begin{itemize}
\item [(1)]
If $\Sigma$ is connected, then $\Omega$ is isometirc to the warped product space
$$
\mathbb{S}^{n} \times [-\theta,\infty)_{t},\,\,\,g=dt^2+(\cosh{t})^2 g_{\mathbb{S}^{n}}.
$$
\item [(2)]
If $\Sigma$ is disconnected,
then $\Omega$ is isometirc to the warped product space
$$
\mathbb{S}^{n} \times [-\theta,\theta]_{t},\,\,\,g=dt^2+(\cosh{t})^2 g_{\mathbb{S}^{n}}.
$$
\end{itemize}
\label{theorem: 1.2}
\end{theorem}

The proofs of Theorems \ref{theorem: 1.3} and \ref{theorem: 1.2} are based on the following
(slightly general) proposition, which provides the warped product structure of $\Omega$ without 
assuming the above curvature conditions.

\begin{proposition}
Let $(\Omega^{n+1}, g)$ be an $(n+1)$-dimensional ($n\geq 2$) smooth complete connected Riemannian manifold with compact boundary $\Sigma$. Assume that $f$ is a non-constant solution of equation (\ref{eqn-1}).
Set $\Omega_0=\{p\in \Omega\,| \,f(p)=0\}$. Then, we have
\begin{itemize}
\item [(1)]
If $f$ is constant on some boundary component, then $\Omega_0$ is a closed manifold. In particular, we also have the following 
\begin{itemize}
\item [(a)]
If $$
\sup\limits_{\Omega} |f|= \infty,
  $$ 
then $\Sigma$ is connected and $\Omega$ is isometirc to the warped product space
$$
\Omega_0 \times [-\theta,\infty)_{t},\,\,\,g=dt^2+(\cosh{t})^2 g_{|{\Omega_0}}.
$$
\item [(b)]
If $$
\sup\limits_{\Omega} |f|< \infty,
  $$ 
then $\Omega$ is isometirc to the warped product space
$$
\Omega_0 \times [-\theta,\theta]_{t},\,\,\,g=dt^2+(\cosh{t})^2 g_{|{\Omega_0}}.
$$
\end{itemize}
\item [(2)]
If $f$ is non-constant on any boundary component, then $\Sigma$ is connected, $\Omega_0$ is compact and $\Omega$ is isometirc to a $Z_2$-symmetric bounded domain in the warped product space
  $$
  \widehat{\Omega}=\Omega_0 \times (-\infty,\infty)_{t}, \,\,g=dt^2 + (\cosh{t})^2 g_{|{\Omega_0}},
  $$
  which is bounded by the graph functions $\pm \phi$, where $\phi \in \, C^{\infty}(\Omega^{\circ}_0) \cap C(\Omega_0)$ satisfies 
  $$
  \frac{\cosh{\phi}}{\sqrt{1+(\cosh{\phi})^{-2}|\nabla_{\Omega_0} \phi|_{g_{|\Omega_0}}^2}}=c\sinh{\phi}, \,\,\,{\text{in}} \,\,\Omega^{\circ}_0,
  $$
  $$
  \phi>0, \,\,\, in \,\,\Omega^{\circ}_0,
  $$
  and
  $$
  \phi=0, \,\,\, on \,\,\partial \Omega_0.
  $$
\end{itemize}
\label{theorem: 1.1}
\end{proposition}

\begin{remark}
The two cases in Proposition \ref{theorem: 1.1} are completely independent; 
in other words, if $f$ is non-constant on any boundary component, 
it won't be a small perturbation of a certain constant, cf. Proposition 2.1.
\end{remark}

From Proposition \ref{theorem: 1.1}, we can see that for the first case, $\Omega_0$ is 
a closed manifold which is conformal to the boundary component(s) and the second fundamental 
form $h$ on $\Sigma$ satisfies $h=\frac{1}{c}g_{|\Sigma}$; while for the second case, $\Omega_0$ is 
a compact manifold with boundary. In general, the (finer) structure of $\Omega_0$ is unclear;
so, in Theorems \ref{theorem: 1.3} and \ref{theorem: 1.2}, we use the curvature assumptions 
discussed above to determine its structure (and then the structure of $\Omega$). 
Here, we also point out that Theorem \ref{theorem: 1.3} corresponds to the second case 
in Proposition \ref{theorem: 1.1} and Theorem \ref{theorem: 1.2} corresponds to the first one.

The idea and method in Theorem \ref{theorem: 1.3} also applies 
to standard sphere $\mathbb{S}^{n+1}$ by using the Obata type equation $\nabla^2 f+ fg=0$.

\begin{theorem}
Let $(\Omega^{n+1},g)$ be an $(n+1)$-dimensional ($n\geq 2$) smooth complete connected Riemannian manifold with compact boundary $\Sigma$. Assume that there exist a non-constant function $f \in \, C^{\infty}(\Omega)$ and a constant $c>0$ such that
\begin{equation}\label{eqn-2}
\begin{cases}
\begin{aligned}
&\nabla^2 f+ fg=0,& {\text{in}}\,\Omega,\\
&f_{\nu}-cf=0,&{\text{on}} \, \Sigma.
\end{aligned}
\end{cases}
\end{equation}
If the Ricci curvature of $\Omega$ satisfies Ric$^{\Omega}\geq n$ and the mean curvature of $\Sigma$ is bounded from below by $c$, then $\Omega$ is isometric to a geodesic ball
of radius $\tan^{-1} {(\frac{1}{c})}$ in the standard sphere $\mathbb{S}^{n+1}$.
\label{theorem: 1.4}
\end{theorem}

We also point out that Theorem \ref{theorem: 1.4} is an extension of Corollary 4.4 in \cite{Liu}.

The paper is organized as follows. Section 2 gives some necessary preliminaries, including some basic definitions and some known results which are needed later. Section 3 concentrates on the Obata type equation (\ref{eqn-1}) and the proof of Proposition \ref{theorem: 1.1}. In Section 4, we prove Theorems \ref{theorem: 1.3}, \ref{theorem: 1.2} and \ref{theorem: 1.4} and give some corollaries.

\section{Preliminaries}

This section mainly introduces some basic definitions and some known results 
which are needed in the later proofs.

Let $(\Omega^{n+1},g)$ be an $(n + 1)$-dimensional smooth complete connected Riemannian manifold with compact boundary $\partial \Omega=\Sigma$ 
and $g|_{\Sigma}$ the restricted metric on $\Sigma$; 
denote by $\langle \cdot, \cdot \rangle$ 
the inner product on $\Omega$ as well as $\Sigma$. Denote by
$\nabla^{\Omega}$, $\nabla$, $\Delta$, and $\nabla^2$ the connection, 
the gradient, the Laplacian, and the Hessian on $\Omega$ respectively, while by $\nabla_{\Sigma}$ and $\Delta_{\Sigma}$ the gradient 
and the Laplacian on $\Sigma$ respectively.
Let $\nu$ be the unit outward normal of $\Sigma$; denote by $h$, $A^{\nu}$, and $H$ the second fundamental form, the Weingarten transformation, and the mean curvature of $\Sigma$ with respect 
to $\nu$ respectively, here 
$$
h(X, Y)= -\langle \nabla^{\Omega}_X Y, \nu \rangle, 
~~\langle A^{\nu}(X), Y\rangle=h(X, Y),
$$ 
and 
$$
H={\frac{\mbox{tr}_g h}{n}}.
$$
The principal curvatures of $\Sigma$ are defined to be the
eigenvalues of $h$ and $A^{\nu}$. 
Let $R^{\Omega}$ be the curvature tensor of $\Omega$, i.e., 
for tangent vectors $X, Y, Z, W$,
$$
R^{\Omega}(X,Y,Z,W)=\langle \nabla^{\Omega}_X \nabla^{\Omega}_Y Z- \nabla^{\Omega}_Y \nabla^{\Omega}_X Z-\nabla^{\Omega}_{[X,Y]} Z, W\rangle;
$$
$\mbox{Ric}^{\Omega}$ be the Ricci curvature tensor of $\Omega$. 
Let $dV$ and $dv$ be the canonical volume element of $\Omega$ and $\Sigma$ respectively.

Let $f \in \, C^{\infty}(\Omega)$ satisfy the equation (\ref{eqn-1}),
which is non-constant and 
$$
\Omega_{a}=\{x\in \Omega~|~f(x)=a\}.
$$

Now, we give some basic facts related to the Obata type equations. 

\begin{proposition}
   Let $f \in  C^{\infty}(\Omega)$ satisfy the equation (\ref{eqn-1}) 
   which is non-constant, then there exists a constant $A>0$ such that 
   $$
   |\nabla f|^2-f^2=A.
   $$
   \label{prop: 2.1}
\end{proposition}

\begin{proof}
A direct calculation shows that, for any tangent vector (field) $X$,
$$
X(|\nabla f|^2-f^2)=2\nabla ^2 f (X,\nabla f)-2f\langle X,\nabla f \rangle=0.
$$
Therefore, $|\nabla f|^2-f^2$ is constant. Since $f$ is non-constant, $f_{\nu}=cf$, and $c>1$, we have $|\nabla f|^2-f^2>0$. 
Then the conclusion follows. 
\end{proof}

Proposition \ref{prop: 2.1} clearly implies that $f$ has no critical points
in $\Omega$. Without lose of generality, we always assume that $A=1$ 
in the following. Here, we also record the following fact, for its proof
one can refer to \cite{chen2020obata}.

\begin{proposition}
   Let $f \in  C^{\infty}(\Omega)$ be a non-constant function which satisfies equation (\ref{eqn-1}), then the
integral curves of $\frac{\nabla f}{|\nabla f|}$ are geodesics.
   \label{prop: 2.2}
\end{proposition}


For later convenience, we here introduce the warped product structure 
of the space forms and the corresponding equations of geodesic spheres. These facts will be used in the proofs of Theorems \ref{theorem: 1.3} 
and \ref{theorem: 1.4}. For sake of clarity, we first see
the Euclidean case.

Let $\mathbb{R}^{n+1}=\{(x^1,x^2,\cdots,x^{n+1})~|~x^{i}\in\mathbb{R}\}$
and $g_{\mathbb{R}^{n+1}}= (dx^1)^2 + (dx^2)^2+\cdots +(dx^{n+1})^2$,
and identify $\mathbb{R}^{n}$ with 
$\{(x^1, x^2,\cdots, x^{n+1})\in \mathbb{R}^{n+1}~|~ x^1=0\}.$ 
Obviously, the Euclidean space $\mathbb{R}^{n+1}$ is isometric to 
the warped product space
$$
\mathbb{R}^{n}\times (-\infty,\infty)_t,\,\,\,g=dt^2+ g_{\mathbb{R}^{n}}.
$$
Let $p \in \mathbb{R}^{n} \subset \mathbb{R}^{n+1}$ and $B_{p}^{n+1}(r)$ be the (closed) ball of radius $r$ centered at $p$ in $\mathbb{R}^{n+1}$. Then, $B_{p}^{n}(r)=B_{p}^{n+1}(r)\cap \mathbb{R}^{n}$ is 
the (closed) ball of radius $r$ centered at $p$ in $\mathbb{R}^n$.  
Now, we consider $B_{p}^{n+1}(r)$ as a bounded domain 
in the above warped product space 
$\mathbb{R}^{n} \times (-\infty,\infty)_t$. A direct calculation 
shows that there exists a unique non-negative function 
$\phi \in C^{\infty}((B_{p}^{n}(r))^{\circ})\cap C(B_{p}^{n}(r))$ 
such that
$$
\partial B_{p}^{n+1}(r) = \{(x,\pm\phi(x))~|~x \in B_{p}^{n}(r)\}.
$$
In particular, $\phi$ satisfies the following equation
$$\phi^2(x)+d^2(x,p)=r^2,$$
where $d$ is the distance function in $\mathbb{R}^{n}$.

Now, we consider the hyperbolic space $\mathbb{H}^{n+1}$. 
We use the model of the upper half-space, i.e. 
$$
\mathbb{H}^{n+1}=\{(x^1,x^2,\cdots,\\x^{n+1})~|~x^{n+1} >0\}
$$ 
and 
$$
g_{\mathbb{H}^{n+1}}= \frac{(dx^1)^2 + (dx^2)^2+\cdots +(dx^{n+1})^2}{(x^{n+1})^2};
$$ 
identify $\mathbb{H}^{n}$ with 
$\{(x^1, x^2, \cdots, x^{n+1})\in\mathbb{H}^{n+1}~|~x^1=0\}$. 
For any point $p=(0, x^2, \cdots, x^{n+1})\in \mathbb{H}^{n}$, 
a direct calculation shows that 
$$
\gamma_p(t)=(\frac{\sinh{t}}{\cosh{t}}x^{n+1}, x^2, \cdots, x^n, \frac{x^{n+1}}{\cosh{t}})
$$ 
is the (arc length parameter) minimizing geodesic satisfying
$\gamma_{p}(0)=p$ and $\gamma_{p}^{\prime}(0) \perp T_{p} \mathbb{H}^{n}$. 
Then, one has the following smooth map
$$
\Psi:\mathbb{H}^{n} \times (-\infty,\infty)_t \rightarrow \mathbb{H}^{n+1}
$$
by setting $\Psi(p,t)=\gamma_p(t)$. It is easy to see 
that $\Psi$ is a diffeomorphism and 
the pull-back metric $\Psi^*g_{\mathbb{H}^{n+1}}$
on $\mathbb{H}^{n} \times (-\infty,\infty)$ can be written as
$$
g=dt^2 + (\cosh{t})^2g_{\mathbb{H}^{n}}.
$$
Thus, the hyperbolic space $\mathbb{H}^{n+1}$ 
is actually isometric to the following warped product 
space
$$
\mathbb{H}^{n} \times (-\infty,\infty)_t,\,\,\,g=dt^2 + (\cosh{t})^2g_{\mathbb{H}^{n}};
$$
this is the warped product structure of $\mathbb{H}^{n+1}$.

Let $p \in \mathbb{H}^{n} \subset \mathbb{H}^{n+1}$ and 
$B_{p}^{n+1}(r)$ be the closed geodesic ball of radius $r$ 
centered at $p$ in $\mathbb{H}^{n+1}$. 
Let $B_{p}^{n}(r)=B_{p}^{n+1}(r)\cap \mathbb{H}^{n}$; obviously $B_{p}^{n}(r)$ is the closed geodesic ball 
of radius $r$ centered at $p$ in $\mathbb{H}^{n}$. 
By the above discussion, we can also consider $B_{p}^{n+1}(r)$ 
as a bound domain in the product space 
$\mathbb{H}^{n} \times (-\infty,\infty)_t$.
Set 
$$
\partial B_{p+}^{n+1}(r) = \partial B_{p}^{n+1}(r) \cap (\mathbb{H}^{n} \times [0,\infty)_t).
$$
Then, it is easy to see that there exists a non-negative function
$\phi \in C^{\infty}((B_{p}^{n}(r))^{\circ})\cap C(B_{p}^{n}(r))$
so that $\partial B_{p+}^{n+1}(r)$ can be written as 
$$
\{(x,\phi(x))\in \mathbb{H}^{n} \times [0,\infty)_t~|~
x \in B_{p}^{n}(r)\}.
$$
For a point $x\in B_{p}^{n}(r)$ together with the points $p$ and 
$(x,\phi(x))$, one has the corresponding geodesic triangle 
in $\mathbb{H}^{n+1}$; by the law of cosine in $\mathbb{H}^{n+1}$, 
we have
$$
\cosh{r}=\cosh{\phi(x)}\cosh{d(x,p)},
$$
that is,
$$
\frac{\tanh^2{d(x,p)}}{\tanh^2{r
}}+\frac{\sinh^2{\phi(x)}}{\sinh^2{r}}=1,
$$
where $d$ is the distance function in $\mathbb{H}^{n}$.
Similarly, one can consider the lower half part
$$
\partial B_{p-}^{n+1}(r) = \partial B_{p}^{n+1}(r) \cap (\mathbb{H}^{n} \times (-\infty,0]_t).
$$ 
Thus, we have the following
\begin{proposition}
There exists some 
$\phi \in C^{\infty}((B_{p}^{n}(r))^{\circ})\cap C(B_{p}^{n}(r))$ 
which is non-negative and satisfies
$$
\frac{\tanh^2{d(x,p)}}{\tanh^2{r}}+
\frac{\sinh^2{\phi(x)}}{\sinh^2{r}}=1
$$
(where $d$ is the distance function in $\mathbb{H}^{n}$), so that
the geodesic sphere $\partial B_{p}^{n+1}(r)$ can be written as
$$
\{(x,\pm\phi(x))~|~x \in B_{p}^{n}(r)\}
$$
in the product space $\mathbb{H}^{n} \times (-\infty,\infty)_t$.
\label{prop:2.3}
\end{proposition}

By a similar argument, the standard sphere $\mathbb{S}^{n+1}$ can be 
isometrically written as the warped product space
$$
\mathbb{S}^{n} \times (-\frac{\pi}{2},\frac{\pi}{2})_t,\,\,\,g=dt^2 + (\cos{t})^2g_{\mathbb{S}^{n}}.
$$
Let $p\in\mathbb{S}^{n}\subset\mathbb{S}^{n+1}$ and $B_{p}^{n+1}(r)$ 
be the closed geodesic ball of radius $r<\frac{\pi}{2}$ centered at $p$ in $\mathbb{S}^{n+1}$. Let $B_{p}^{n}(r)=B_{p}^{n+1}(r)\cap \mathbb{S}^{n}$. We then have the following

\begin{proposition}
For $r<\frac{\pi}{2}$, there exist some 
$\phi \in C^{\infty}((B_{p}^{n}(r))^{\circ})\cap C(B_{p}^{n}(r))$ 
which is non-negative and satisfies
$$
\frac{\tan^2{d(x,p)}}{\tan^2{r
}}+\frac{\sin^2{\phi(x)}}{\sin^2{r}}=1
$$
(where $d$ is the distance function in $\mathbb{S}^{n}$), 
so that the geodesic sphere $\partial B_{p}^{n+1}(r)$ can be written as 
$$
\{(x,\pm\phi(x))~|~x\in B_{p}^{n}(r)\}
$$
in the product space 
$\mathbb{S}^{n} \times (-\frac{\pi}{2},\frac{\pi}{2})_t$.
\label{prop:2.4}
\end{proposition}

\section{The warped product structure of $\Omega$ and 
proof of Proposition \ref{theorem: 1.1}}

Now, let us concentrate on Proposition \ref{theorem: 1.1}; the method of 
the proof is similar to that of Theorem 1.3 in \cite{chen2020obata}.
We first consider the case that $f$ 
is constant on some boundary component. 

Since $|\nabla f|^2-f^2=1$ and $f_\nu =cf$, we have 
    \begin{equation*}
        |\nabla_{\Sigma} f|^2+(c^2-1)f^2=1,\,\,\mbox{on} \,\,\Sigma.
    \end{equation*}
We then conclude that for any $p \in \Sigma$,
$$-\sinh{\theta}=\frac{-1}{\sqrt{c^2-1}}\leq f(p) \leq \frac{1}{\sqrt{c^2-1}}=\sinh{\theta}.$$
Let $S \subset \Sigma$ be the boundary component such that $f|_{S}$ is constant. Then we can assume that 
$$f|_{S}=\frac{-1}{\sqrt{c^2-1}}=-\sinh{\theta}.$$
Then we have
$$
f_{\nu}|_{S}=-\cosh{\theta}.
$$
$\forall \, p \in S$, we consider the integral curve 
$\gamma_{p}(t)$ of $\frac{\nabla f}{|\nabla f|}$ starting 
from $p$. By Proposition \ref{prop: 2.2}, we know that 
$\gamma_{p}$ is a geodesic. Then a direct calculation shows
$$
f(\gamma_{p}(t))=\sinh{(t-\theta)}.
$$
Define $h: S \rightarrow \mathbb{R}\cup \{\infty\}$ as following:
$$h(p)=\sup\limits \{t>0 ~|~\gamma_{p}|_{(0,t)}\subset \Omega^{\circ}\}.$$
By the definition of $h$, we know that if $h(p)=\infty$, then for any $t>0$,
$\gamma_{p}(t) \in \Omega^{\circ}$ and if $h(p)<\infty$, $\gamma_{p}$ will meet $\Sigma \setminus S$ at the time $h(p)$.
In fact, $h$ has the following properties:
\begin{proposition}
  Let $p \in S$. 
  
  (1) If $h(p) = \infty$, there exists a neighborhood $U_p$ of $p$ in $S$ such that for any point $q \in U_p$, $h(q)=\infty$.
  
(2) If $h(p) < \infty$, there exists a neighborhood $U_p$ of $p$ in $S$ such that for any point $q \in U_p$, $h(q)<\infty$.

(3) If $$\sup\limits_{S} h <\infty,$$ then $h$ satisfies 
$$h\equiv 2\theta.$$
  \label{proposition:3.1}
\end{proposition}
\begin{proof}
    (1) and (2) follow from the definition of boundary, so we omit the proof. Now we prove (3). Since 
    $$\sup\limits_{S} h <\infty,$$
    we know that $\gamma_{p}(t)$ will meet $\Sigma \setminus S$ at the time $h(p)$. 
    Since the boundary is a smooth manifold, we conclude that $h$ is smooth. 
    Let $h(p_0) = \min_{p \in S} h(p)$ and $\gamma_{p_0} $ be 
the corresponding geodesic starting from $p_0$ such 
that $\gamma_{p_0}(h(p_0)) \in \Sigma \setminus S$.
We then know that 
$\gamma_{p_0}^{\prime}(h(p_0))\perp T_{\gamma_{p_0}(h(p_0))}\Sigma $. Therefore, we have
$$\frac{d}{dt}f(\gamma_{p_0}(t))|_{t=h(p_0)}=cf(\gamma_{p_0}(h(p_0))),
$$
that is,
$$
\cosh{(h(p_0)-\theta)}=c\sinh{(h(p_0)-\theta)},
$$
which shows 
$$
h(p_0)=2\theta
$$
and
$$
f(\gamma_{p_0}(h(p_0)))=\sinh{\theta}=\frac{1}{\sqrt{c^2-1}}.
$$
Then for any $p \in S$, we have
$$
h(p) = 2\theta.
$$
\end{proof}

\begin{proposition}
Let $S \subset \Sigma$ be the boundary component such that $f|_{S}=-\sinh{\theta}$. 
If $$
\sup\limits_{\Omega} |f|= \infty,
$$ 
then $\Omega$ is isometirc to the
warped product space
$$
S\times [0,\infty)_{t},\,\,\,g=dt^2+\frac{(\cosh{(t-\theta))^2}}{(\cosh{\theta)^2}} g|_{S}.
$$
\label{proposition:3.2}
\end{proposition}
\begin{proof}
    Since $$
\sup\limits_{\Omega} |f|= \infty,
$$ 
there must be a point $p \in S$ such that $h(p)=\infty$. Then by Proposition \ref{proposition:3.1}, we know that for any $q \in S$, $h(q)=\infty$. By Morse theory,
we conclude that $\Omega$ is isometric to the warped product space
$$S \times [0,\infty)_t,$$
with metric
$$
g=dt^2+g(t),
$$
where $g(t)$ is a family of metrics on $S$ and $g(0)=g|_{S}$. Moreover, 
$$
f=\sinh{(t-\theta)}.
$$
Since $\nabla^2 f -fg =0$, a direct calculation shows that
$$
\frac{1}{2}f^{\prime}(t)g^{\prime}(t)-f(t)g(t)=0,
$$
that is,
$$
\frac{1}{2}\cosh{(t-\theta)}g^{\prime}(t)=\sinh{(t-\theta)}g(t),
$$
which shows that 
$$
g(t)=\frac{(\cosh{(t-\theta))^2}}{(\cosh{\theta)^2}} g|_{S}.
$$
We then finish the proof.
\end{proof}

\begin{proposition}
Let $S \subset \Sigma$ be the boundary component such that $f|_{S}=-\sinh{\theta}$. 
If $$
\sup\limits_{\Omega} |f|< \infty,
$$ 
then $\Omega$ is isometirc to the
warped product space
$$
S\times [0,2\theta],\,\,\,g=dt^2+\frac{(\cosh{(t-\theta))^2}}{(\cosh{\theta)^2}} g|_{S}.
$$
\label{proposition:3.2.1}
\end{proposition}
\begin{proof}
    Since $$
\sup\limits_{\Omega} |f|< \infty,
$$ 
we know that 
$$\sup\limits_{S} h <\infty.$$ 

By Proposition \ref{proposition:3.1} (3) and Morse theory,
we conclude that $\Omega$ is isometric to the warped product space
$$S \times [0,2\theta]_t,$$
with metric
$$
g=dt^2+g(t),
$$
where $g(t)$ is a family of metrics on $S$ and $g(0)=g|_{S}$. By a similar caculation, we know that
$$
g(t)=\frac{(\cosh{(t-\theta))^2}}{(\cosh{\theta)^2}} g|_{S}.
$$
We then finish the proof.
\end{proof}

\begin{proof1}
The first case in Proposition \ref{theorem: 1.1} follows from Propositions \ref{proposition:3.1}, \ref{proposition:3.2} and \ref{proposition:3.2.1}. 
\end{proof1}

Next, we concentrate on the second case in Proposition \ref{theorem: 1.1}. Since $f$ is non-constant on any boundary component and $|\nabla f|\not =0$, we know that $\Omega_0$ is a manifold of dimension $n$ with compact boundary $\partial \Omega_0$. Moreover, we also have $\Omega^{\circ}_0=\Omega^{\circ} \cap \Omega_0$ and $\partial \Omega_0= \Sigma \cap \Omega_0$.

\begin{proposition}
(1) For any $p \in \Omega$ with $f(p)>0$, the integral curve of $-\frac{\nabla f}{|\nabla f|}$ starting at $p$ will meet $\Omega^{\circ}_0$ before it reaches $\Sigma$. For any $p\in\Omega$ with $f(p)<0$, the integral curve of $\frac{\nabla f}{|\nabla f|}$ starting at $p$ will meet $\Omega^{\circ}_0$ before it reaches $\Sigma$.
    
(2) For $0<a_1<a_2$, the integral curve of $-\frac{\nabla f}{|\nabla f|}$ defines an injective map from $\Omega_{a_2}$ to $\Omega_{a_1}$. For $0>a_1>a_2$, the integral curve of $\frac{\nabla f}{|\nabla f|}$ defines an injective map from $\Omega_{a_2}$ to $\Omega_{a_1}$.

  \label{proposition:3.3}
\end{proposition}

\begin{proof}
Take a point $p$ with $f(p)\not =0$. By a similar proof of Lemma 4.3 in \cite{chen2020obata}, we know that the corresponding integral curve (geodesic) starting at $p$ will meet $\Omega_0$ before it reaches $\Sigma$. 

Now, we prove that the integral curve meets $\Omega^{\circ}_0$ before it reaches $\Sigma$. In fact, we only need to prove that for any point $q \in \partial \Omega_0$, the geodesic $\gamma_{q}(t)$ which satisfies $\gamma_{q}(0)=q $ and $\gamma_{q}^{\prime}(0)=\frac{\nabla f}{|\nabla f|}(q)$ or $-\frac{\nabla f}{|\nabla f|}(q)$ is not contained in $\Omega^{\circ}$ when $t$ is close to $0$. Since $f(q)=0$ and $f_{\nu}= cf$, we know that $\frac{\nabla f}{|\nabla f|}(q)$ is tangent to $T_{q}\Sigma$. Let $\{e_1,e_2,\cdots,e_{n}\}$ be an orthonormal frame near $q$ in $T\Sigma$ such that $e_1(q) = \frac{\nabla f}{|\nabla f|}(q)$, then at the point $q$, we have
$$ c(e_1 f)=e_1 (\nu f)= \nabla^{2} f (e_1,\nu) + \nabla^{\Omega} _{e_1} \nu f=\nabla^{\Omega} _{e_1} \nu f,$$
which shows 
$$h(e_1,e_1)=c>0.$$
We then know that the geodesic $\gamma_{q}(t)$ discuessed above is not contained in $\Omega^{\circ}$ when $t$ is close to $0$.

Obviously, (2) follows from (1).
\end{proof}

For any $a \in \mathbb{R}$ such that $\Omega_a \not= \varnothing$, by Proposition \ref{proposition:3.3}, we know that the integral curve discussed above defines an injective map
$$
\Psi_{a}: \Omega_a\rightarrow \Omega_0.
$$
Therefore, $\Omega$ can be considered as a domain with boundary of the warped product space $(\widehat{\Omega},g)$, where
$$
\widehat{\Omega}=\Omega_0 \times (-\infty,\infty)_{t}, \,\,g=dt^2 + (\cosh{t})^2 g|_{\Omega_0}
$$
and
$$
f =\sinh{t}.
$$
Since $\Omega$ is connetced and $\Omega_0 \subset \Omega$, we also know that $\Omega_0$ is connected.

For any $x \in \Omega^{\circ}_0$, we consider the integral curve $\gamma_x(t)$ of $\frac{\nabla f}{|\nabla f|}$ starting at $x$. Similar to the definition of the function $h$ discussed above, we define 
$$\phi: \Omega_0^{\circ} \rightarrow \mathbb{R} \cup \{\infty\}$$
as following
$$\phi(x)=\sup\{t>0 ~|~ \gamma_x|_{(0,t)}\subset \Omega^{\circ}\}.$$
We then know that if $\phi(x)<\infty$, $\gamma_x$ will meet $\Sigma$ at the time $\phi(x)$.
\begin{proposition}
For any $x \in \Omega^{\circ}_0$, 
$$0< \phi(x) \leq \theta.$$
Moreover, the function $\phi$ is smooth.
  \label{proposition:3.3.1}
\end{proposition}

\begin{proof}
 Obviously, $\phi$ is positive. By the definition of boundary, for any $x \in \Omega^{\circ}_0$, if $\phi(x) = \infty$, then there exists a neighborhood $U_x$ of $x$ in $\Omega^{\circ}_0$ such that for any point $y \in U_x$, $\phi(y)=\infty$. Similarly, if $\phi(x) < \infty$, there exists a neighborhood $U_x$ of $x$ in $\Omega^{\circ}_0$ such that for any point $y \in U_x$, $\phi(y)<\infty$. Then we know that the set $\{x \in \Omega^{\circ}_0 ~|~ \phi(x) <\infty\}$ is open and closed. For any $l \in (0,\theta]$, we know that there exists a point $p \in \Sigma$ such that $f(p) = \sinh{l}.$ By Proposition \ref{proposition:3.3}, there exists a point $x \in \Omega^{\circ}_0$ such that $\phi(x)=l<\infty$. Then we conclude that 
 $$\{x \in \Omega^{\circ}_0 ~|~ \phi(x) <\infty\}=\Omega^{\circ}_0.$$
 Since $|\nabla_{\Sigma} f|^2+(c^2-1)f^2=1$ and $f=\sinh{t}$, we know that
 $$\phi(x) \leq \theta.$$
 Since the boundary $\Sigma$ is a smooth manifold, the function 
$\phi$ is smooth in $\Omega_0^{\circ}$.
\end{proof}

Now, we extend $\phi$ to $\partial \Omega_0$ and define the function as following
\begin{equation*}
\phi(x)=\begin{cases}
\begin{aligned}
&\phi(x), & x \in \Omega_0^{\circ},\\
&0, & x \in \partial \Omega_0.
\end{aligned}
\end{cases}
\end{equation*}
We show that $\phi$ is continuous. If there exists a sequence $\{x_m\} \subset \Omega_0$ such that $x_m \rightarrow x \in \partial \Omega_0$, but $\phi(x_m)$ does not converge to $0$, then we can find a subsequence (denote also by $\{x_m\}$) such that $x_m \rightarrow x \in \partial \Omega_0$ and $\phi(x_m) \rightarrow a >0$. In other words, the sequence $\{(x_m,\phi(x_m))\}$ converges to the point $(x,a) \in \Sigma$. However, by Proposition \ref{proposition:3.3}, we conclude $x \in \Omega_0^{\circ}$, which is a contradiction. Therefore, $\phi$ is continuous.

Now we consider the boundary $\Sigma$. In fact, we have the following
\begin{proposition}

Let $\Sigma_{+}=\{p \in \Sigma~|~f(p)\geq 0\}$. Then 
$$
\Sigma_{+}=\{(x,\phi(x)) \in \widehat{\Omega} ~|~ x\in \Omega_0\},
$$
where $\phi \in \, C^{\infty}(\Omega^{\circ}_0)\cap C(\Omega_0)$ 
is non-negative and satisfies 
$$
\frac{\cosh{\phi}}{\sqrt{1+(\cosh{\phi})^{-2}|\nabla_{\Omega_0} \phi|_{g|_{\Omega_0}}^2}}=c\sinh{\phi}, \,\,\,{\text{in}} \,\,\Omega^{\circ}_0,
$$
$$\phi>0, \,\,\, {\text{in}} \,\,\Omega^{\circ}_0,
$$
and
$$
\phi=0, \,\,\, {\text{on}} \,\,\partial \Omega_0.
$$
\label{proposition:3.4}
\end{proposition}

\begin{proof}
Let $S=\{p \in \Sigma ~|~ f(p)> 0\}$. Obviously, $\Sigma_{+} = S \cup \partial \Omega_0$. For any $p \in S$, by Proposition \ref{proposition:3.3}, the integral curve of $-\frac{\nabla f}{|\nabla f|}$ starting at $p$ will meet $\Omega^{\circ}_0$ before it reaches $\Sigma$. In addition, by Proposition \ref{proposition:3.3.1}, for any $x \in \Omega^{\circ}_0$, the integral curve of $\frac{\nabla f}{|\nabla f|}$ starting at $x$ will meet $S$ at some point $p$. Let $\gamma:[0,\phi(x)] \rightarrow \Omega$ be the integral curve (geodesic), where $\gamma(0)=x$, $\gamma(\phi(x))=p$ and $\gamma^{\prime} = \frac{\nabla f}{|\nabla f|}$. Obviously, $\gamma([0,\phi(x))) \subset \Omega^{\circ}$ and $\gamma$ will stop at the point $p$.
By the above discussion, we also conclude 
$$
S=\{(x,\phi(x))\in \widehat{\Omega} ~| ~x\in \Omega_0^{\circ}\}.
$$
For any point $p=(x,\phi(x)) \in S$, the outward unit normal of $\Omega$ is 
$$
\nu = \frac{(-\nabla_{\widehat{\Omega}_{\phi(x)}}\phi,1)}{\sqrt{1+(\cosh{\phi})^{-2}|\nabla_{\Omega_0} \phi|_{g|_{\Omega_0}}^2}},
$$
where $\widehat{\Omega}_{\phi(x)} = \Omega_0 \times \{\phi(x)\}$ (consider $\phi$ as a function on $\widehat{\Omega}_{\phi(x)}$), $g|_{\widehat{\Omega}_{\phi(x)}}= (\cosh{\phi(x)})^2g|_{\Omega_0} $.
Since $f=\sinh{t}$ and $\frac{\partial f}{\partial \nu} =cf$, we then have 
\begin{align*}
c\sinh{\phi}=\frac{\partial f}{\partial\nu}=\nabla f\cdot\nu
=&(0,\cosh{\phi})\cdot\nu \\
=& \frac{\cosh{\phi}}{\sqrt{1+(\cosh{\phi})^{-2}|\nabla_{\Omega_0} \phi|_{g|_{\Omega_0}}^2}}.
\end{align*}
Since $\Sigma_{+} = S \cup \partial \Omega_0$, we then know that 
$$
\Sigma_{+}=\{(x,\phi(x)) \in \widehat{\Omega} ~|~ x\in \Omega_0\}.
$$
\end{proof}

Proposition \ref{proposition:3.4} also implies the following

\begin{proposition}
$\Sigma_{+}$ is homeomorphic to $\Omega_0$ and 
thus $\Omega_0$ is compact and $\Sigma_{+}$ is connected.
\label{proposition:3.5}
\end{proposition}

\begin{proof}
We define a continuous map 
$$
\Psi: \Sigma_{+} \rightarrow \Omega_0,
$$
where $\Psi((x,\phi(x))=x$.
By Proposition \ref{proposition:3.4}, we know that $\Psi$ is bijective and $\Psi^{-1}$ is continuous. Therefore, $\Sigma_{+}$ is homeomorphic to $\Omega_0$. Since $\Sigma_{+}$ is compact and $\Omega_0$ is connected, we know that $\Omega_0$ is compact and $\Sigma_{+}$ is connected.
\end{proof}

Let $\Sigma_{-}=\{p \in \Sigma | f(p)\leq 0\}$.
Similarly, we also have

\begin{proposition}
(1) $\Sigma_{-}=\{(x,-\psi(x)) \in \widehat{\Omega} ~|~ x\in \Omega_0\},$
where $\psi \in \, C^{\infty}(\Omega^{\circ}_0) \cap C( \Omega_0)$ is non-negative and satisfies
$$
\frac{\cosh{\psi}}{\sqrt{1+(\cosh{\psi})^{-2}|\nabla_{\Omega_0} \psi|_{g|_{\Omega_0}}^2}}=c\sinh{\psi}, \,\,\,{\text{in}} \,\,\Omega^{\circ}_0,
$$
$$
\psi>0, \,\,\, {\text{in}} \,\,\Omega^{\circ}_0,
$$
and
$$
\psi=0, \,\,\, {\text{on}} \,\,\partial \Omega_0.
$$

(2) $\Sigma_{-}$ is homeomorphic to $\Omega_0$ and thus $\Sigma_{-}$ is connected.
\label{proposition:3.6}
\end{proposition}

By Propositions \ref{proposition:3.5} and \ref{proposition:3.6}, we conclude that $\Sigma$ is connected and $0\leq \psi \leq\theta$.
Now, we only need to prove $\phi = \psi$.

\begin{proposition}
There exists a unique solution to the following equation
$$
\frac{\cosh{\phi}}{\sqrt{1+(\cosh{\phi})^{-2}|\nabla_{\Omega_0} \phi|_{g|_{\Omega_0}}^2}}=c\sinh{\phi}, \,\,\,{\text{in}} \,\,\Omega^{\circ}_0,
$$
such that 
$$
0<\phi \leq\theta, \,\,\,{\text{in}} \,\,\Omega^{\circ}_0,
$$
and
$$
\phi=0, \,\,\, {\text{on}} \,\,\partial \Omega_0.
$$
\label{proposition:3.7}
\end{proposition}

\begin{proof}
 Assume that $\phi$ and $\psi$ are the solutions and $\phi \not \equiv \psi$. Without loss of generality, we assume
 $$\mathop{\mbox{max}}\limits_{\Omega_0} |\phi-\psi|=\phi(p)-\psi(p)>0,$$
 where $p \in \Omega^{\circ}_0$. Thus 
 $$0< \psi (p) < \phi (p) \leq \theta.$$
 
 A direct calculation shows
 $$
 \nabla _{\Omega_0} (\phi - \psi) \cdot \nabla _{\Omega_0} (\phi + \psi)=\frac{\cosh^{4}{\phi}}{c^2\sinh^{2}{\phi}}-\cosh^{2}{\phi}-\frac{\cosh^{4}{\psi}}{c^2\sinh^{2}{\psi}}+\cosh^{2}{\psi}.
 $$
 It is elementary to show that the function 
 $$
 h(t)=\frac{\cosh^{4}{t}}{c^2\sinh^{2}{t}}-\cosh^{2}{t}
 $$
 is monotonically decreasing when $t \in (0,\theta) $.
 Therefore, at the point $p$, we have 
 $$
 \nabla _{\Omega_0} (\phi - \psi) \cdot \nabla _{\Omega_0} (\phi + \psi)=0
 $$
 and
 $$
 \frac{\cosh^{4}{\phi}}{c^2\sinh^{2}{\phi}}-\cosh^{2}{\phi}-\frac{\cosh^{4}{\psi}}{c^2\sinh^{2}{\psi}}+\cosh^{2}{\psi}<0,
 $$
which is a contradiction.
\end{proof}

\begin{proof5}
The second case in Proposition \ref{theorem: 1.1} follows directly from Propositions \ref{proposition:3.3} -- \ref{proposition:3.7}.
\end{proof5}

\section{The structures of $\Omega_0$ and proofs of Theorems \ref{theorem: 1.3}, \ref{theorem: 1.2} and \ref{theorem: 1.4}}

In this section, we prove Theorems \ref{theorem: 1.3}, \ref{theorem: 1.2} and \ref{theorem: 1.4}. Based on the curvature assumption ($K_1$) and Proposition \ref{theorem: 1.1}, we can directly prove Theorem \ref{theorem: 1.3}.\\

\begin{proof3}
By the curvature assumption ($K_1$) in Theorem \ref{theorem: 1.3}, Theorem 0.1 in \cite{MR3404748} and Lemma 2.1 in \cite{MR3373063},
we know that $\Omega$ is compact and $\Sigma$ is connected.
Then by Proposition \ref{theorem: 1.1} (2), we conclude that $\Omega$ is a $Z_2$-symmetric domain in the warped product space
$$
\widehat{\Omega}=\Omega_0 \times (-\infty,\infty)_{t}, \,\,g=dt^2 + (\cosh{t})^2 g|_{\Omega_0},
$$
which is bounded by the graph functions $\pm \phi$, where $\phi\in C^{\infty}(\Omega^{\circ}_0) \cap C( \Omega_0)$ satisfies 
$$
\frac{\cosh{\phi}}{\sqrt{1+(\cosh{\phi})^{-2}|\nabla_{\Omega_0} \phi|_{g|_{\Omega_0}}^2}}=c\sinh{\phi}, \,\,\,{\text{in}} \,\,\Omega^{\circ}_0,
$$
$$
\phi>0, \,\,\, {\text{in}} \,\,\Omega^{\circ}_0
$$
and
$$
\phi=0, \,\,\, {\text{on}} \,\,\partial \Omega_0.
$$
In particular, in the coordinate of $\Omega_0 \times (-\infty,\infty)_{t}$, we have $f=\sinh{t}$ and $\nabla f = (\cosh{t}) \frac{\partial}{\partial t}$.
  
Let $\widehat{\Omega}_{t}=\Omega_0 \times \{t\}$, where $t\in  (-\infty,\infty)$ (obviously, $\widehat{\Omega}_{0}=\Omega_0$). We first consider the second fundamental form with respect to $\frac{\partial}{\partial t}$, denoted by $h_{\widehat{\Omega}_{t}}$. Let $(x_1,x_2,\cdots,x_n)$ be a coordinate in $\Omega_0$, then the metric has the form
$$
g=dt^2 + (\cosh{t})^2 g_{ij}dx^i dx^j,
$$
where $g|_{\Omega_0}=g_{ij}dx^i dx^j$.
A direct calculation shows that
$$
\begin{aligned}
h_{\widehat{\Omega}_{t}} (\frac{\partial }{\partial x_i},\frac{\partial }{\partial x_j})
&= -\langle \nabla^{\widehat{\Omega}}_{\frac{\partial }{\partial x_i}} \frac{\partial }{\partial x_j} , \frac{\partial }{\partial t}\rangle\\ 
&= -\Gamma_{ij}^t\\
&= \frac{\sinh t}{\cosh t}\langle\frac{\partial }{\partial x_i}, \frac{\partial }{\partial x_j}\rangle,
\end{aligned}
$$
which implies that
$$
h_{\widehat{\Omega}_{t}} = \frac{\sinh t}{\cosh t}g|_{\widehat{\Omega}_{t}}.
$$
Therefore, $\Omega_0$ is a totally geodesic hypersurface in $\widehat{\Omega}$ 
with unit normal $\frac{\partial}{\partial t}$.
    
Clearly, the graph function $\phi$ has the following properties.
First, $\phi \in  [0, \tanh^{-1}{(\frac{1}{c})}]$ and $\phi|_{\partial \Omega_0}=0$. Second, the set $\{x \in  \Omega_0~|~ \phi(x)=\tanh^{-1}{(\frac{1}{c})}\}$ (denoted by $A$) is a compact subset in $\Omega_0$.

Define 
$$
v=\tanh^{-1}({\frac{1}{c}\sqrt{1-(c^2-1)\sinh^2{\phi}}}).
$$
Then $v$ also has the following properties. First, $v \in  [0, \tanh^{-1}{(\frac{1}{c})}]$ and $v|_{\partial \Omega_0}=\tanh^{-1}({\frac{1}{c}})$. Second, $v$ is smooth at any $x$ with $v(x) \in \, (0, \tanh^{-1}{(\frac{1}{c})})$. Third, $\{x\in  \Omega_0 ~|~ v(x)=0\}=A$. In addition, a direct calculation shows that $|\nabla _{\Omega_0} v|\equiv 1$ on $\Omega_0- (\partial \Omega_0 \cup A)$. 

We now consider the level set of $v$. Set
$$
T_t=\{x\in \Omega_0~|~ v(x)=t\},
$$
where $t \in  [0,\tanh^{-1}({\frac{1}{c}})]$. Obviously, $T_0= A$ and $T_{\tanh^{-1}({\frac{1}{c}})}=\partial\Omega_0$.
We now prove
$$
t\leq d(A,T_t),
$$
where $d$ is the distance function in $\Omega_0$.
Fix $t \in (0,\tanh^{-1}({\frac{1}{c}})]$. Let $\widehat{\gamma}:[0,l] \rightarrow \Omega_0$ be a minimizing geodesic realizing the distance $d(A,T_t)$ with arc length parameter such that $\widehat{\gamma}(0)\in A$, $\widehat{\gamma}(l)\in T_t$, and $l=d(A,T_t)$. We then have
$$
\begin{aligned}
t
&= \lim_{\epsilon \rightarrow 0^+} v(\widehat{\gamma}(s))|_{\epsilon}^{l-\epsilon}\\ 
&= \lim_{\epsilon \rightarrow 0^+} \int_{\epsilon}^{l-\epsilon} \frac{d}{ds} (v(\widehat{\gamma}(s)))ds\\
&= \lim_{\epsilon \rightarrow 0^+} \int_{\epsilon}^{l-\epsilon} \langle \nabla_{\Omega_0} v, \gamma^{\prime} \rangle ds\\
&\leq l\\
&=d(A,T_t),
\end{aligned}
$$
where $\langle \nabla_{\Omega_0} v, \gamma^{\prime} \rangle \leq |\nabla _{\Omega_0} v|=1$. Therefore, we conclude that $t\leq d(A,T_t)$. In particular, $\tanh^{-1}({\frac{1}{c}}) \leq d(A,T_{\tanh^{-1}({\frac{1}{c}})}) = d(A,\partial \Omega_0)$.
        
Now, we prove that $\Omega_0$ is a geodesic ball
of radius $\tanh^{-1} {(\frac{1}{c})}$ in the hyperbolic space $\mathbb{H}^{n}$ by the curvature assumptions.

First, let $e_1=\frac{\partial}{\partial t}$ and $\{e_2,e_3,\cdots,e_{n+1}\}$ be an orthonormal frame in $T\Omega_0$, then we know that $\{e_1,e_2,\cdots,e_{n+1}\}$ is an orthonormal frame in $T\widehat{\Omega}$.  Since $\Omega_0$ is totally geodesic, by the Gauss equation,
we have 
$$
R^{\Omega_0}(e_i,e_j,e_k,e_l)=R^{\widehat{\Omega}}(e_i,e_j,e_k,e_l),
$$
where $i,j,k,l \in \, \{2,3,\cdots,n+1\}$.
In addition, by the Ricci identity, we obtain
$$
\begin{aligned}
f_{k}\delta_{ij}-f_{j}\delta_{ik}
&= f_{ijk}-f_{ikj}\\
&= -\sum_{p=1}^{n+1}R^{\widehat{\Omega}}(e_p,e_i,e_j,e_k)f_p\\ 
&= -R^{\widehat{\Omega}}(e_1,e_i,e_j,e_k)f_1,
\end{aligned}
$$
which implies
$$
R^{\widehat{\Omega}}(e_1,e_j,e_1,e_j)=1,
$$
where $j\in\,\{2,3,\cdots,n+1\}$.
Therefore, we can deduce
$$
\begin{aligned}
\mbox{Ric}^{\Omega_0}(e_i,e_i)
&= -\sum_{j=2,j\not=i}^{n+1} R^{\Omega_0}(e_i,e_j,e_i,e_j)\\
&= -\sum_{j=2,j\not=i}^{n+1}R^{\widehat{\Omega}}(e_i,e_j,e_i,e_j)\\ 
&= -\sum_{j=1,j\not=i}^{n+1}R^{\widehat{\Omega}}(e_i,e_j,e_i,e_j)+1\\
&\geq -(n-1).
\end{aligned}$$
        
Next, we prove that the second fundamental form $h_{\partial \Omega_0}$ of $\partial \Omega_0$ in $\Omega_0$ satisfies $h_{\partial \Omega_0}=h|_{\partial \Omega_0}$, where $h$ is the the second fundamental form of $\Sigma$ in $\Omega$.  In fact, $\forall \, P \in \, \partial \Omega_0$, we have $f_{\nu}(P)=cf(P)=0$, that is, $\langle \nabla f ,\nu \rangle (P)=0$, and we know that $\nabla f(P) \in \,T_P \Sigma$. Let $\{e_1,e_2,\cdots,e_{n}\}$ be an orthonormal frame in $T\Sigma$ near the point $P$ such that $e_1(P)=\nabla f(P)$. Since $\Omega_0$ is totally geodesic with constant unit normal $\frac{\partial}{\partial t}$ and $\nabla f(P)= \frac{\partial}{\partial t}(P)$, we know that $\{e_2,e_3,\cdots,e_{n}\}$ is an orthonormal frame for $\partial \Omega_0$ at the point $P$ and $\nu$ is the unit outward normal of $\partial \Omega_0$ in $\Omega_0$. In addition, for $i \in \, \{2,3,\cdots,n\}$ we have 
$$
0=\nabla ^2 f (e_i,\nu)=e_i(\nu f)-\nabla^{\widehat{\Omega}}_{e_i}\nu f=cf_i-\sum_{j=1}^{n} h_{ij}f_j,
$$
which implies that $h_{11}=c$ and $h_{1i}=0$ for $i \in \{2,3,\cdots,n\}$. So $e_1$ is a principal
direction of $\Sigma$ at $P$ corresponding to the principal curvature $c$.
Now, let $\{e_2,e_3,\cdots,e_{n}\}$ be an orthonormal frame in $T\partial \Omega_0$, since $\Omega_0$ is totally geodesic with constant unit normal $\nabla f$, we then have
$$
h_{\partial \Omega_0}(e_i,e_j)=-\langle \nabla ^{\Omega_0} _{e_i} e_j, \nu \rangle=-\langle \nabla ^{\widehat{\Omega}} _{e_i} e_j, \nu \rangle=h(e_i,e_j),
$$
where $i,j \in\, \{2,3,\cdots,n\}$. Therefore, we conclude that $h_{\partial \Omega_0}=h|_{\partial \Omega_0}$ and we know that the mean curvature of $\partial \Omega_0$ in $\Omega_0$, denoted by $H_{\partial \Omega_0}$, satisfies 
\begin{align*}
H_{\partial \Omega_0}=\frac{1}{n-1} \mbox{tr}_{g_{\partial \Omega_0}} h_{\partial \Omega_0}=&\frac{1}{n-1} \mbox{tr}_{g_{\partial \Omega_0}} h|_{\partial \Omega_0}\\
\geq& \frac{1}{n-1} (nc-c)\\
=&c.
\end{align*}
Thus, by the fact that $\tanh^{-1}({\frac{1}{c}}) \leq d(A,\partial \Omega_0)$ and Theorem 0.3 in \cite{MR3404748}, we conclude that $\Omega_0$ is isometric to a geodesic ball
of radius $\tanh^{-1} {(\frac{1}{c})}$ in the hyperbolic space $\mathbb{H}^{n}$ and $A$ consists of a single point $x_0$, which is the center of $\Omega_0$. For convenience, we just assume that $\Omega_0$ is the geodesic ball in $\mathbb{H}^{n}$. Then we know that $\widehat{\Omega}$ is a domain in the hyperbolic space $\mathbb{H}^{n+1}$, and so do $\Omega$.

Now, let $t \in \, (0,\tanh^{-1}{(\frac{1}{c})})$, we prove that $T_{t}=S_{t}$, where $S_t$ is the geodesic sphere in $\Omega_0$ of
radius $t$ centered at $x_0$. On the one hand, we have known that $t\leq d(x_0,T_t)$. On the other hand, we can use the similar method discussed above to prove that $\tanh^{-1}{(\frac{1}{c})}-t \leq d(T_t,\partial \Omega_0)$, and we then have $T_{t}=S_{t}$, that is,
$$v(x)=d(x,x_0).$$
Since 
$$
v=\tanh^{-1}({\frac{1}{c}\sqrt{1-(c^2-1)\sinh^2{\phi}}}),
$$
we have
$$\frac{\tanh^2{d(x,x_0)}}{\tanh^2{\theta
}}+\frac{\sinh^2{\phi(x)}}{\sinh^2{\theta}}=1,$$
where $\theta=\tanh^{-1}(\frac{1}{c}).$
By Proposition \ref{prop:2.3}, we know that $\Sigma$ is a geodesic sphere in the hyperbolic space $\mathbb{H}^{n+1}$ of
radius $\tanh^{-1}(\frac{1}{c})$ centered at $x_0$ and thus $\Omega$ is a geodesic ball
of radius $\tanh^{-1} {(\frac{1}{c})}$ in the hyperbolic space $\mathbb{H}^{n+1}$.
\end{proof3}

Now we concentrate on the proof of Theorem \ref{theorem: 1.2}. We first show that Theorem \ref{theorem: 1.2} corresponds to the first case in Proposition \ref{theorem: 1.1} by the assumption ($K_2$).

\begin{proposition}
Let $(\Omega^{n+1},g)$ and $f$ be as in Theorem \ref{theorem: 1.2}. Then $f$ is constant on any boundary component. 
\label{prop:4.0.5}
\end{proposition}

\begin{proof}
    In fact, we only need to show that there exists a boundary component $S$ such that $f|_{S}$ is constant. Assume that $f$ is non-constant on any boundary component, then we know that $\Sigma$ is connected. From the proof of Theorem \ref{theorem: 1.3}, we also conclude that for any $P \in \Sigma$ with $f(P)=0$, $\nabla f(P)$ is a principal direction corresponding to the principal curvature $c$, which is a contradiction (since we have $h<cg|_{\Sigma}$).
\end{proof}

By Proposition \ref{prop:4.0.5}, we know that Theorem \ref{theorem: 1.2} corresponds to the first case in Proposition \ref{theorem: 1.1} and $\Omega$ is isometirc to the warped product space

$$
\Omega_0 \times [-\theta,\infty)_{t},\,\,\,g=dt^2+(\cosh{t})^2 g_{|{\Omega_0}}
$$
or
$$
\Omega_0 \times [-\theta,\theta]_{t},\,\,\,g=dt^2+(\cosh{t})^2 g_{|{\Omega_0}}.
$$

Now we can consider the structure of $\Omega_0$.

\begin{proposition}
Let $(\Omega^{n+1},g)$ and $f$ be as in Theorem \ref{theorem: 1.2}. Let $S$ be the boundary component which satisfies the assumption ($K_2$). Then we have
$$
\mbox{Ric}^{S}\geq (n-1)(1-\frac{1}{c^2})>0.
$$
\label{prop:4.1}
\end{proposition}

\begin{proof}
Take an orthonormal frame $\{e_1,e_2,\cdots,e_n\}$ in $TS$. By Gauss equation, for $i=2,\cdots,n$, we have
$$
\begin{aligned}
R^{S}(e_1,e_i,e_1,e_i)
&= R^{\Omega}(e_1,e_i,e_1,e_i)-h(e_1,e_1) h(e_i,e_i) +  h(e_1,e_i) h(e_1,e_i).
\end{aligned}
$$
A direct caculation shows that $h=\frac{1}{c}g|_{S}$, we then know 
$$
R^{S}(e_1,e_i,e_1,e_i)
= R^{\Omega}(e_1,e_i,e_1,e_i)-\frac{1}{c^2}.
$$
Since $S$ satisfies the assumption (K), for any subset 
$\{i_1,i_2,\cdots,i_{k-1}\} \subset \{2,\cdots,n\}$, 
we have
$$
\begin{aligned}
\sum_{j=1}^{k-1} -R^{S}(e_1,e_{i_j},e_1,e_{i_j})
&\geq (k-1)(1-\frac{2}{c^2}) + (k-1)\frac{1}{c^2}\\
&=(k-1)(1-\frac{1}{c^2}).
\end{aligned}
$$
Obviously, the number of the inequalities discussed above is $C^{k-1}_{n-1}$. By summing up all these inequalities, we have 
$$
\begin{aligned}
\sum_{j=2}^{n} -C^{k-2}_{n-2}R^{S}(e_1,e_{j},e_1,e_{j})
&\geq C^{k-1}_{n-1}(k-1)(1-\frac{1}{c^2}),
\end{aligned}
$$
that is, 
$$
\begin{aligned}
\mbox{Ric}^{S}(e_1,e_1)
&\geq(n-1)(1-\frac{1}{c^2})>0.
\end{aligned}
$$
\end{proof}

\begin{proposition}
Let $S$ be the boundary component discussed above. If we further assume that the diameter $d$ of $S$ satisfies
$$
d\geq \frac{c}{\sqrt{c^2-1}} \pi,
$$
then $\Omega_0$ is isometric to the standard 
sphere $\mathbb{S}^n$.
\label{prop:4.2}
\end{proposition}

\begin{proof}
By Bonnet-Myers theorem and Cheng's theorem, we conclude that $S$ is isometric to a sphere of radius $\cosh{\theta}$ ($c=\coth \theta$). Since $\Omega_0$ is conformal to $S$, then a direct calculation shows that $\Omega_0$ is isometric to the
standard sphere $\mathbb{S}^n$.
\end{proof}

\begin{proof2}
Theorem \ref{theorem: 1.2} follows from Proposition \ref{theorem: 1.1} and Propositions \ref{prop:4.0.5}, \ref{prop:4.1} and \ref{prop:4.2}.
\end{proof2}

Next, we prove Theorem \ref{theorem: 1.4}. In this case, we consider the following Obata type equation
\begin{equation*}
\begin{cases}
\begin{aligned}
&\nabla^2 f+ fg=0,& {\text{in}} \,\Omega,\\
&f_{\nu}-cf=0, &{\text{on}} \, \Sigma,
\end{aligned}
\end{cases}
\end{equation*}
where $c$ is a positive constant. For convenience, 
we still assume that
$$
|\nabla f|^2+f^2=1.
$$
By Theorem 0.1 in \cite{MR3404748} and the curvature assumptions in Theorem \ref{theorem: 1.4}, we know that $\Omega$ is compact.
Set $\Omega_0=\{p \in \Omega~|~ f(p)=0\}$, by Theorem 1.3 in \cite{chen2020obata} and the curvature assumptions in Theorem \ref{theorem: 1.4}, we conclude that $\Sigma$ is connected (see also \cite{ichida1981riemannian})
and $\Omega$ is a $Z_2$-symmetric domain in the warped product space
$$
\widehat{\Omega}=\Omega_0 \times (-\frac{\pi}{2},\frac{\pi}{2})_{t}, \,\,g=dt^2 + (\cos{t})^2 g|_{\Omega_0},
$$
which is bounded by the graph functions $\pm \phi$, where $\phi \in \, C^{\infty}(\Omega^{\circ}_0) \cap C(\Omega_0)$ satisfies 
$$
\frac{\cos{\phi}}{\sqrt{1+(\cos{\phi})^{-2}|\nabla_{\Omega_0} \phi|_{g|_{\Omega_0}}^2}}=c\sin{\phi}, \,\,\,{\text{in}} \,\,\Omega^{\circ}_0,
$$
$$
\phi>0, \,\,\, {\text{in}} \,\,\Omega^{\circ}_0,
$$
and
$$
\phi=0, \,\,\, {\text{on}} \,\,\partial \Omega_0.
$$
In particular, $f=\sin{t}$.

Let $\widehat{\Omega}_{t}=\Omega_0 \times \{t\}$, where $t\in  (-\frac{\pi}{2},\frac{\pi}{2})$. We first consider the second fundamental form with respect to $\frac{\partial}{\partial t}$, denoted by $h_{\widehat{\Omega}_{t}}$. By a similar method in the proof of Theorem \ref{theorem: 1.3}, we have
$$
h_{\widehat{\Omega}_{t}} = -\frac{\sin{t}}{\cos{t}}g|_{\widehat{\Omega}_{t}}.
$$
Therefore, $h_{\widehat{\Omega}_{t}}$ is negative definite when $t>0$ and $\Omega_0$ is a totally geodesic hypersurface in $\widehat{\Omega}$ with unit normal $\frac{\partial}{\partial t}$.

As for the graph function $\phi$. Obviously, $\phi \in \, [0, \tan^{-1}{(\frac{1}{c})}]$ and $\phi|_{\partial \Omega_0}=0$. Here we can provide another way to prove that the set $\{x \in \, \Omega_0~|~ \phi(x)=\tan^{-1}{(\frac{1}{c})}\}$ consists of a single point, denoted by $x_0$, which is different from the method in Theorem \ref{theorem: 1.3}.  To see this, we first prove Lemma \ref{lemma:4.3}. The idea of this lemma comes from Lemma 2.1 in \cite{do2001rigidity}. 

\begin{lemma}
Let $\Omega$ and f be as in Theorem \ref{theorem: 1.4}. 
Assume that $p\in\Sigma$ is a critical point of $f|_{\Sigma}$, then at the point $p$, the second fundamental form $h$ satisfies
$$
h=cg|_{\Sigma}.
$$
\label{lemma:4.3}
\end{lemma}

\begin{proof}
For convenience, let $z=f|_{\Sigma}$. Since $|\nabla f|^2+f^2=1$ and $f_{\nu}=cf$, on $\Sigma$, we have
$$
1=|\nabla f|^2+f^2=|\nabla_{\Sigma} z|^2+(f_{\nu})^2+z^2=|\nabla_{\Sigma} z|^2+(c^2+1)z^2
$$
and we know that $z(p)\not=0$.
  
Let $\{e_1,e_2,\cdots,e_{n+1}\}$ be an orthonormal frame near the boundary satisfying $e_{n+1}|_{\Sigma}=\nu$. $\forall \, i,j \in \,\{1,2,\cdots,n\}$, we have
$$
\begin{aligned}
            z_{ij}
            &= e_i(e_jz)-\nabla^{\Sigma}_{e_i}e_j z\\ 
            &= e_i(e_jz)-\nabla_{e_i}e_j z -h_{ij}f_{\nu}\\
            &= \nabla^2 f(e_i,e_j) -h_{ij}f_{\nu}\\
            &=-z\delta_{ij}-ch_{ij}z\\
            &=-(ch_{ij}+\delta_{ij})z
\end{aligned}
$$
and
$$
\begin{aligned}
cz_{i}
&= e_i(e_{n+1}z)-\nabla_{e_i}e_{n+1} f+\nabla_{e_i}e_{n+1}f\\ 
&= \nabla^2 f(e_i,e_{n+1})+\sum_{j=1}^{n} h_{ij}z_{j}\\
&= \sum_{j=1}^{n} h_{ij}z_{j}.
\end{aligned}
$$
By taking covariant differentiation with respect to $e_k$ ($k\in \,\{1,2,\cdots,n\}$) at $p$, we have
$$
\begin{aligned}
-c(ch_{ik}+\delta_{ik})z
&= cz_{ik}\\ 
&= \sum_{j=1}^{n} e_{k}h_{ij}z_{j}+\sum_{j=1}^{n} h_{ij}z_{jk}\\
&= \sum_{j=1}^{n} -h_{ij}(ch_{jk}+\delta_{jk})z,
\end{aligned}
$$
which implies
$$
c^2h(p)+cI=ch^2(p)+h(p),
$$
where $h(p)$ is the $n \times n$ matrix $(h_{ij})$ and $I$ is the identity matrix.
We then have
$$
(ch(p)+I)(h(p)-cI)=0.
$$
Since the mean curvature $H \geq c$, we obtain that
$$
h(p)=cI.
$$
This finishes the proof of the lemma.
\end{proof}

Now, we can prove that $\{x \in  \Omega_0~|~ \phi(x)=\tan^{-1}{(\frac{1}{c})}\} =\{x_0\}$.

\begin{proposition}
The set $\{x \in  \Omega_0~|~ \phi(x)=\tan^{-1}{(\frac{1}{c})}\}$ consists of a single point, denoted by $x_0$.
\label{prop:4.4}
\end{proposition}

\begin{proof}
Since $f=\sin{t}$, we only need to prove that
$\{x\in \Sigma~|~ f(x)=\frac{1}{\sqrt{c^2+1}}\}$ consists of a single point. Since
$|\nabla_{\Sigma} f|^2+(c^2+1)f^2=1$, we know that $f|_{\Sigma}$ is a transnormal function on $\Sigma$. Then by Theorem A in \cite{MR0901710}, we know that $\{x \in  \Sigma| f(x)=\frac{1}{\sqrt{c^2+1}}\}$ is a submanifold of $\Sigma$, denoted by $M$. In particular, we know that any points in $M$ are critical points of $f|_{\Sigma}$ and the outward unit normal $\nu$ satisfies $\nu=\frac{\partial}{\partial t}$ on $M$. Assume that dim$M \geq1$, then we can consider a smooth curve $\gamma :[0,l] \rightarrow M$ with arc length parameter. Obviously, $\gamma$ is a smooth curve in $\Sigma$ and we also know that $\gamma([0,l]) \subset \widehat{\Omega}_{\tan^{-1}{(\frac{1}{c})}}$. We then have
$$
\begin{aligned}
            h (\gamma^{\prime},\gamma^{\prime})
            &= -\langle \nabla^{\widehat{\Omega}}_{\gamma^{\prime}} \gamma^{\prime} , \nu\rangle\\ 
            &= -\langle \nabla^{\widehat{\Omega}}_{\gamma^{\prime}} \gamma^{\prime} , \frac{\partial }{\partial t}\rangle\\
            &=h_{\widehat{\Omega}_{\tan^{-1}{(\frac{1}{c})}}}(\gamma^{\prime},\gamma^{\prime})\\
            &=-\frac{1}{c}\\
            &<0.
\end{aligned}
$$
However, by Lemma \ref{lemma:4.3}, we have    
$$ 
h (\gamma^{\prime},\gamma^{\prime})=c>0,
$$
which is a contradiction.
Therefore, dim$M =0$. Then by Lemma 2.6 in \cite{chen2020obata}, 
we know that $M$ is connected and consists of a single point. Hence, 
$\{x \in  M_0~|~ \phi(x)=\tan^{-1}{(\frac{1}{c})}\}=\{x_0\}$.
\end{proof}

The remaining part of the proof of Theorem \ref{theorem: 1.4} is similar to that of Theorem \ref{theorem: 1.3}.

\begin{proof4}
Set 
$$
v=\tan^{-1}({\frac{1}{c}\sqrt{1-(c^2+1)\sin^2{\phi}}}),
$$
$v$ then has the following properties\\
(1) $v \in  [0, \tan^{-1}{(\frac{1}{c})}]$ and $v|_{\partial \Omega_0}=\tan^{-1}({\frac{1}{c}})$.\\
(2)  $v$ is smooth at any $x$ with $v(x) \in \, (0, \tan^{-1}{(\frac{1}{c})})$. \\
(3)  $\{x\in\Omega_0 | v(x)=0\}=\{x_0\}$.\\
(4)  On $\Omega_0- (\partial \Omega_0 \cup \{x_0\})$, $|\nabla _{\Omega_0} v|\equiv 1$. 

By a similar argument in the proof of Theorem \ref{theorem: 1.3}, we conclude that $\Omega_0$ is a geodesic ball
of radius $\tan^{-1} {(\frac{1}{c})}$ in the standard sphere $\mathbb{S}^{n}$ centered at $x_0$
and the graph function $\phi$ satisfies
$$
\frac{\tan^2{d(x,x_0)}}{\tan^2{r}}+\frac{\sin^2{\phi(x)}}{\sin^2{r}}=1,
$$
where $r=\tan^{-1}(\frac{1}{c}).$
By Proposition \ref{prop:2.4}, we know that $\Sigma$ is a geodesic sphere in the standard sphere $\mathbb{S}^{n+1}$ of
radius $\tan^{-1}(\frac{1}{c})$ centered at $x_0$ and thus $\Omega$ is a geodesic ball
of radius $\tan^{-1} {(\frac{1}{c})}$ in the standard sphere $\mathbb{S}^{n+1}$.
\end{proof4}

As the end of this section, we give some applications of the Theorems \ref{theorem: 1.3} and \ref{theorem: 1.4} as follows.

Let $(\Omega^{n+1},g)$ be an $(n+1)$-dimensional smooth compact connected Riemannian manifold with smooth boundary $\Sigma$. Assume that the Ricci curvature of $\Omega$ satisfies Ric$^{\Omega}\geq -n$ and the principal curvatures of $\Sigma$ are bounded from below by a positive constant $c>1$. Then we know that $\Sigma$ is connected. Let $u$ be an eigenfunction corresponding to the first nonzero
eigenvalue $\lambda_1$ of the Laplacian on $\Sigma$.

Then we know that the Dirichlet problem 
\begin{equation*}
\begin{cases}
\begin{aligned}
&\Delta f=(n+1)f,& {\text{in}} \,\Omega,\\
&f=u,&{\text{on}} \, \Sigma
\end{aligned}
\end{cases}
\end{equation*}
has a unique solution $f \in C^{\infty}(\Omega)$. We then have the following inequality and rigidity result.

\begin{corollary}
Let $\Omega$ and $f$ be the manifold and smooth function discussed above. If we further assume that the mean curvature of $\Sigma$ is bounded from below by $\frac{\lambda_1 + n}{nc}$. Then we have
$$
c||u||^2_{L^2(\Sigma)}   \geq  (u,f_{\nu})_{L^2(\Sigma)}.
$$
Moreover,  equality holds if and only if $\Omega$ is isometric 
to a geodesic ball of radius $\tanh^{-1} {(\frac{1}{c})}$ 
in the hyperbolic space $\mathbb{H}^{n+1}$.
\label{corollary:4.5}
\end{corollary}

\begin{proof}
Since $u$ is a non-constant eigenfunction, we have
$$
\begin{aligned}
\int_{\Sigma}  |\nabla_{\Sigma}u|^2 dv= -\int_{\Sigma} (\Delta_{\Sigma} u)u dv=\lambda_1 \int_{\Sigma} u^2 dv.
\end{aligned}
$$
By the Reilly-type formula in \cite{qiu2015generalization}, we have
$$
\begin{aligned}
&\int_{\Omega} \bigg ([\Delta f-(n+1)f]^2-|\nabla ^2 f - fg|^2 \bigg) dV\\ 
=& \int_{\Sigma} [2(\Delta_{\Sigma} u) f_{\nu} + nH(f_{\nu})^2 + h(\nabla_{\Sigma}u, \nabla_{\Sigma}u)-2nuf_{\nu}]dv\\
&+\int_{\Omega} [Ric^{\Omega}(\nabla f,\nabla f)+2n |\nabla f|^2 +n(n+1)f^2]dV
\end{aligned}
$$
Since $h\geq cI$ and $H\geq \frac{\lambda_1 + n}{nc}>0$, we conclude that
$$
\begin{aligned}
0&\geq\int_{\Omega} \bigg([\Delta f-(n+1)f]^2-|\nabla ^2 f - fg|^2\bigg)dV\\
&\geq \int_{\Sigma}[ -2\lambda_1 u f_{\nu} + \frac{\lambda_1+n}{c}(f_{\nu})^2 + c\lambda_1u^2-2nuf_{\nu}]dv\\
&+\int_{\Omega} [n |\nabla f|^2 +n(n+1)f^2]dV\\
&= \int_{\Sigma} [\frac{\lambda_1+n}{c}(f_{\nu})^2- (n+2\lambda_1) u f_{\nu} + c\lambda_1u^2]dv,
\end{aligned}
$$
where the last equality is the divergence theorem. Therefore, we have 
$$
0\geq \int_{\Sigma} [(\lambda_1+n)(\frac{f_{\nu}}{\sqrt{c}}-\sqrt{c}u)^2-ncu^2+nuf_{\nu}]dv.
$$
Then we conclude that
$$
c||u||^2_{L^2(\Sigma)}   \geq  (u,f_{\nu})_{L^2(\Sigma)}  .
$$

If $\Omega$ is isometric to a geodesic ball
of radius $\tanh^{-1} {(\frac{1}{c})}$ in the hyperbolic space $\mathbb{H}^{n+1}$, it is easy for us to check that $c||u||^2_{L^2(\Sigma)}   =  (u,f_{\nu})_{L^2(\Sigma)}  .$
Now we assume that $c||u||^2_{L^2(\Sigma)}   =  (u,f_{\nu})_{L^2(\Sigma)}$. Obviously, the above inequalities must take equality sign, we then have
\begin{equation*}
\begin{cases}
\begin{aligned}
&\nabla^2 f- fg=0,& {\text{in}} \,\Omega,\\
&f_{\nu}-cf=0,&{\text{on}} \, \Sigma.
\end{aligned}
\end{cases}
\end{equation*}
By Theorem \ref{theorem: 1.3}, we know that $\Omega$ is isometric to a geodesic ball
of radius $\tanh^{-1} {(\frac{1}{c})}$ in the hyperbolic space $\mathbb{H}^{n+1}$.
\end{proof}

Similarly, we also have the following

Let $(\Omega^{n+1},g)$ be an $(n+1)$-dimensional smooth compact connected Riemannian manifold with smooth boundary $\Sigma$. Assume that the Ricci curvature of $\Omega$ satisfies Ric$^{\Omega}\geq n$ and the principal curvatures of $\Sigma$ are bounded from below by a positive constant $c$. Then $\Sigma$ is connected. Let $u$ be an eigenfunction corresponding to the first nonzero
eigenvalue $\lambda_1$ of the Laplacian on $\Sigma$.

By Theorem 4 in \cite{reilly1977applications}, we know that the Dirichlet problem 
\begin{equation*}
\begin{cases}
\begin{aligned}
&\Delta f+(n+1)f=0,& {\text{in}} \,\Omega,\\
&f=u,&{\text{on}} \, \Sigma
\end{aligned}
\end{cases}
\end{equation*}
has a unique solution $f \in C^{\infty}(\Omega)$. We then have the following inequality and rigidity result.
\begin{corollary}
Let $\Omega$ and $f$ be the manifold and smooth function discussed above. If we further assume that the mean curvature of $\Sigma$ is bounded from below by $\frac{\lambda_1 - n}{nc}>0$, we then have
$$
(u,f_{\nu})_{L^2(\Sigma)} \geq c||u||^2_{L^2(\Sigma)}.
$$
Moreover, equality holds if and only if $\Omega$ is isometric to a geodesic ball
of radius $\tan^{-1} {(\frac{1}{c})}$ in the standard sphere $\mathbb{S}^{n+1}$.
\label{corollary:4.6}
\end{corollary}

The proof of Corollary \ref{corollary:4.6} is similar to that of Corollary \ref{corollary:4.5}, so we omit it.

\bibliographystyle{plain}
\bibliography{ref}

\vskip .5cm
\noindent
Yiwei Liu:\\
School of Mathematical Sciences, Shanghai Jiao Tong University\\
email: lyw201611012@sjtu.edu.cn

\vskip .5cm
\noindent
Yi-Hu Yang:\\
School of Mathematical Sciences, Shanghai Jiao Tong University\\
email: yangyihu@sjtu.edu.cn
\end{document}